\documentclass[11pt]{amsart}
\usepackage{amsfonts}
\usepackage{amsmath, amsthm}

\usepackage{comment}
\usepackage{tensor}

\numberwithin{equation}{section}

\newtheorem*{Theorem}{Theorem}
\newtheorem{Lemma}{Lemma}

\newcommand{\R}{\mathbb{R}}
\renewcommand{\S}{\mathbb{S}}

\newcommand{\del}{\partial}

\renewcommand{\phi}{\varphi}

\renewcommand{\epsilon}{\varepsilon}

\newcommand{\Lip}{\mathrm{Lip}}

\title[A disc maximizes eigenvalues among isoperimetric surfaces]{A disc maximizes Laplace eigenvalues among isoperimetric surfaces of revolution}
\author{Sinan Ariturk}
\date{}

\begin{document}

\begin{abstract}
The Dirichlet eigenvalues of the Laplace-Beltrami operator are larger on a flat disc than on any other surface of revolution immersed in Euclidean space with the same boundary.
\end{abstract}

\maketitle

\section{Introduction}

Let $\Sigma$ be a compact connected immersed surface of revolution in $\R^3$ with one smooth boundary component.
The Euclidean metric on $\R^3$ induces a Riemannian metric on $\Sigma$.
Let $\Delta_\Sigma$ be the corresponding Laplace-Beltrami operator on $\Sigma$.
Denote the Dirichlet eigenvalues of $-\Delta_\Sigma$ by
\[
	0 < \lambda_1(\Sigma) < \lambda_2(\Sigma) \le \lambda_3(\Sigma) \le \ldots
\]
Let $R$ be the radius of the boundary of $\Sigma$, and let $D$ be a disc in $\R^2$ of radius $R$.
Let $\Delta$ be the Laplace operator on $\R^2$, and denote the Dirichlet eigenvalues of $-\Delta$ on $D$ by
\[
	0 < \lambda_1(D) < \lambda_2(D) \le \lambda_3(D) \le \ldots
\]
\begin{Theorem}
If $\Sigma$ is not equal to $D$, then for $j=1,2,3,\ldots$,
\[
	\lambda_j(\Sigma) < \lambda_j(D)
\]
\end{Theorem}

We remark that there are compact connected surfaces, which are not surfaces of revolution, embedded in $\R^3$ whose boundary is a circle of radius $R$ and have first Dirichlet eigenvalue larger than $\lambda_1(D)$.
This can be proven with Berger's variational formulas \cite{Be}.

This problem resonates with the Rayleigh-Faber-Krahn inequality, which states that the flat disc has smaller first Dirichlet eigenvalue than any other domain in $\R^2$ with the same area \cite{F} \cite{K}.
Hersch proved that the canonical metric on $\S^2$ maximizes the first non-zero eigenvalue among metrics with the same area \cite{H}.
Li and Yau showed the canonical metric on $\mathbb{RP}^2$ maximizes the first non-zero eigenvalue among metrics with the same area \cite{LY}.
Nadirashvili proved the same is true for the flat equilateral torus, whose fundamental parallelogram is comprised of two equilateral triangles \cite{N1}.
It is not known if there is such a maximal metric on the Klein bottle, but Jakobson, Nadirashvili, and Polterovich showed there is a critical metric \cite{JNP}.
El Soufi, Giacomini, and Jazar proved this is the only critical metric on the Klein bottle \cite{EGJ}.

As for the second eigenvalue, the Krahn-Szeg\"o inequality states that the union of two discs with the same radius has smaller second Dirichlet eigenvalue than any other domain in $\R^2$ with the same area \cite{K}.
Nadirashvili proved that the union of two round spheres of the same radius has larger second non-zero eigenvalue than any metric on $\S^2$ with the same area \cite{N2}.

It is conjectured that a disc has smaller third Dirichlet eigenvalue than any other planar domain with the same area.
Bucur and Henrot established the existence of a quasi-open set in $\R^2$ which minimizes for the third eigenvalue among sets of prescribed Lebesgue measure \cite{BH}.
This was extended to higher eigenvalues by Bucur \cite{Bu}.

On a compact orientable surface, Yang and Yau obtained upper bounds, depending on the genus, for the first non-zero eigenvalue among metrics of the same area \cite{YY}.
Li and Yau extended these bounds to compact non-orientable surfaces \cite{LY}.
However, Urakawa showed that there are metrics on $\S^3$ with volume one and arbitrarily large first non-zero eigenvalue \cite{U}.
Colbois and Dodziuk extended this to any manifold of dimension three or higher \cite{CD}.

For a closed compact hypersurface in $\R^{n+1}$, Chavel and Reilly obtained upper bounds for the first non-zero eigenvalue in terms of the surface area and the volume of the enclosed domain \cite{Ch,R}.
This was extended to higher eigenvalues by Colbois, El Soufi, and Girouard \cite{CEG}.
Abreu and Freitas proved that for a metric on $\S^2$ which can be isometrically embedded in $\R^3$ as a surface of revolution, the first $\S^1$-invariant eigenvalue is less than the first Dirichlet eigenvalue on a flat disc with half the area \cite{AF}.
Colbois, Dryden, and El Soufi extended this to $O(n)$-invariant metrics on $\S^n$ which can be isometrically embedded in $\R^{n+1}$ as hypersurfaces of revolution \cite{CDE}.

We conclude this section by reformulating the theorem.
Fix a plane in $\R^3$ containing the axis of symmetry of $\Sigma$.
Identify $\R^2$ with this plane isometrically in such a way that the axis of symmetry is identified with
\[
	\{ (x,y) \in \R^2 : x = 0 \}
\]
Define
\[
	\R^2_+ = \{ (x,y) \in \R^2 : x \ge 0 \}
\]
We may assume $\del \Sigma$ intersects $\R^2_+$ at the point $(R,0)$.
Let $L$ be the length of the meridian $\Sigma \cap \R^2_+$.
Let $\alpha: [0,L] \to \R^2_+$ be a regular, arc-length parametrization of $\Sigma \cap \R^2_+$ with $\alpha(0)=(R,0)$.
Write $\alpha=(F_\alpha,G_\alpha)$.
Note that $F_\alpha(L)=0$ and $F_\alpha$ is positive over $[0,L)$.

Let $C_0^1(0,L)$ be the set of functions $w:[0,L] \to \R$ which are continuously differentiable and vanish at zero.
For a non-negative integer $k$ and a positive integer $n$, define
\[
	\lambda_{k,n}(\alpha) = \min_W \max_{w \in W} \frac{\int_0^L | w' |^2 F_\alpha + \frac{k^2 w^2}{F_\alpha} \,dt}{\int_0^L w^2 F_\alpha \,dt}
\]
Here the minimum is taken over all $n$-dimensional subspaces $W$ of $C_0^1(0,L)$.
We remark that
\[
	\bigg\{ \lambda_j(\Sigma) \bigg\} = \bigg\{ \lambda_{k,n}(\alpha) \bigg\}
\]
Moreover, if we count $\lambda_{k,n}(\alpha)$ twice for $k \neq 0$, then the values occur with the same multiplicity.
Define $\omega:[0,R] \to \R^2_+$ by
\[
	\omega(t)=(R-t,0)
\]
Define $\lambda_{k,n}(\omega)$ similarly to $\lambda_{k,n}(\alpha)$.
Then
\[
	\bigg\{ \lambda_j(D) \bigg\} = \bigg\{ \lambda_{k,n}(\omega) \bigg\}
\]
Again, if we count $\lambda_{k,n}(\omega)$ twice for $k \neq 0$, then the values occur with the same multiplicity.
Now to prove the theorem, it suffices to prove the following lemma.

\begin{Lemma}
\label{kn}
If $\alpha$ does not equal $\omega$, then for any non-negative integer $k$ and any positive integer $n$,
\[
	\lambda_{k,n}(\alpha) < \lambda_{k,n}(\omega)
\]
\end{Lemma}

To prove this, we define a neighborhood of the boundary $\del \R^2_+$ and treat the segments of the curve outside and inside of this neighborhood seperately.
For the exterior segment, we simply project $\alpha$ orthogonally onto $\omega$ and observe that this increases the eigenvalue.
For the interior segment, we unroll the curve to $\omega$ and see that this increases the eigenvalue as well.

\section{Proof}

We first extend the definition of the functionals $\lambda_{k,n}$ to Lipschitz curves.
Let $[a,b]$ be a finite, closed interval and let $\psi:[a,b] \to \R^2_+$ be a Lipschitz curve.
Write $\psi=(F_\psi,G_\psi)$.
Assume that $F_\psi$ is positive over $[a,b)$.
Let $\Lip_0(a,b)$ be the set of continuous functions $w:[a,b) \to \R$ which vanish at $a$ and are Lipschitz over $[a,c]$ for every $c$ in $(a,b)$.
For a non-negative integer $k$ and a positive integer $n$, define
\[
	\lambda_{k,n}(\psi) = \inf_W \max_{w \in W} \frac{\int_a^b \frac{| w' |^2 F_\psi}{| \psi' |} + \frac{k^2 w^2 | \psi' |}{F_\psi} \,dt}{\int_a^b w^2 F_\psi | \psi' | \,dt}
\]
Here the infimum is taken over all $n$-dimensional subspaces $W$ of $\Lip_0(a,b)$.
Let $H_0^1(\psi,k)$ be the set of continuous functions $w:[a,b) \to \R$ which vanish at $a$ and have a weak derivative such that
\[
	\int_a^b \frac{| w' |^2 F_\psi}{| \psi' |} + \frac{k^2 w^2 | \psi' |}{F_\psi} \,dt < \infty
\]

In the following lemma, we note that if $\psi$ is a regular piecewise continuously differentiable curve which meets the axis transversally, then the infimum in the defintion of the functionals $\lambda_{k,n}$ is attained.

\begin{Lemma}
\label{efs}
Let $\psi: [a,b] \to \R^2_+$ be a piecewise continuously differentiable curve.
Assume there is a positive constant $c$ such that for all $t$ in $[a,b]$,
\[
	| \psi'(t) | \ge c
\]
Write $\psi=(F_\psi, G_\psi)$.
Assume that $F_\psi$ is positive over $[a,b)$.
Assume that $F_\psi(b)=0$ and $F_\psi'(b)<0$.
Let $k$ be a non-negative integer.
Then there are functions
\[
	\phi_{k,1}, \phi_{k,2}, \phi_{k,3}, \ldots
\]
which form an orthonormal basis of $H_0^1(\psi,k)$ such that, for any positive integer $n$,
\[
	\lambda_{k,n}(\psi) = \frac{\int_a^b \frac{| \phi_{k,n}' |^2 F_\psi}{| \psi' |} + \frac{k^2 \phi_{k,n}^2 | \psi' |}{F_\psi} \,dt}{\int_a^b \phi_{k,n}^2 F_\psi | \psi' | \,dt}
\]
Each function $\phi_{k,n}$ has exactly $n-1$ roots in $(a,b)$ and satisfies the following equation weakly:
\[
	\bigg( \frac{ F_\psi \phi_{k,n}' } {|\psi'|} \bigg)' = \frac{k^2 | \psi' | \phi_{k,n}}{ F_\psi} - \lambda_{k,n}(\psi) F_\psi | \psi' | \phi_{k,n}
\]
Also,
\[
	\lambda_{k,1}(\psi) < \lambda_{k,2}(\psi) < \lambda_{k,3}(\psi) < \ldots
\]
\end{Lemma}

We omit the proof which is standard and refer to Gilbarg and Trudinger \cite{GT} and Zettl \cite{Z}.

Now fix a non-negative integer $K$ and a positive integer $N$, for the remainder of the article.
Let
\[
	\mu=\frac{K}{\sqrt{\lambda_{K,N}(\omega)}}
\]
The inequality $\mu < R$ is a basic fact about Bessel functions \cite{W}.
Let $\alpha$ be as defined in the introduction, and let
\[
	A = \min \bigg\{ t \in [0,L] : F_\alpha(t) = \mu \bigg\}
\]
Define $\beta:[0,L] \to \R^2_+$ to be a piecewise continuously differentiable function such that $\beta(0)=(R,0)$ and
\[
	\beta'(t) =
	\begin{cases}
		(F_\alpha'(t),0) & t \in [0, A) \\
		(F_\alpha'(t),G_\alpha'(t)) & t \in (A, L] \\
	\end{cases} 
\]

\begin{Lemma}
\label{ab}
Assume $\alpha$ is not equal to $\beta$ and $\lambda_{K,N}(\alpha) \ge \lambda_{K,N}(\omega)$.
Then
\[
	\lambda_{K,N}(\alpha) < \lambda_{K,N}(\beta)
\]
\end{Lemma}

\begin{proof}
Fix a number $p$ in $(0,1)$.
Define $\alpha_p:[0,L] \to \R^2_+$ to be a regular piecewise continuously differentiable curve such that $\alpha_p(0)=(R,0)$ and
\[
	\alpha_p'(t) =
	\begin{cases}
		(F_\alpha'(t),pG_\alpha'(t)) & t \in [0,A) \\
		(F_\alpha'(t),G_\alpha'(t)) & t \in (A,L] \\
	\end{cases} 
\]
We first show that
\[
	\lambda_{K,N}(\alpha) < \lambda_{K,N}(\alpha_p)
\]
By Lemma \ref{efs}, there is a $N$-dimensional subspace $\Phi$ of $H_0^1(\alpha_p,K)$ such that
\[
	\lambda_{K,N}(\alpha_p) = \max_{w \in \Phi} \frac{\int_0^L \frac{| w' |^2 F_\alpha}{| \alpha_p' |} + \frac{K^2 w^2 | \alpha_p' |}{F_\alpha} \,dt}{\int_0^L w^2 F_\alpha | \alpha_p' | \,dt}
\]
Moreover $\Phi$ is contained in $\Lip_0(0,L)$ and the maximum over $\Phi$ is only attained by scalar multiples of a function $\phi_{K,N}$ which has exactly $N-1$ roots in $(0,L)$.
Let $v$ be a function in $\Phi$ such that
\[
	\frac{\int_0^L \frac{| v' |^2 F_\alpha}{| \alpha' |} + \frac{K^2 v^2 | \alpha' |}{F_\alpha} \,dt}{\int_0^L v^2 F_\alpha | \alpha' | \,dt}
		= \max_{w \in \Phi} \frac{\int_0^L \frac{| w' |^2 F_\alpha}{| \alpha' |} + \frac{K^2 w^2 | \alpha' |}{F_\alpha} \,dt}{\int_0^L w^2 F_\alpha | \alpha' | \,dt}
\]
Note this quantity is at least $\lambda_{K,N}(\alpha)$, which is at least $\lambda_{K,N}(\omega)$.
It follows that
\[
	\frac{\int_0^L \frac{| v' |^2 F_\alpha}{| \alpha' |} + \frac{K^2 v^2 | \alpha' |}{F_\alpha} \,dt}{\int_0^L v^2 F_\alpha | \alpha' | \,dt}
		\le \frac{\int_0^L \frac{| v' |^2 F_\alpha}{| \alpha_p' |} + \frac{K^2 v^2 | \alpha_p' |}{F_\alpha} \,dt}{\int_0^L v^2 F_\alpha | \alpha_p' | \,dt}
\]
If equality holds, then $v$ must vanish on a set of positive measure.
In either case, we obtain
\[
	\lambda_{K,N}(\alpha) \le \frac{\int_0^L \frac{| v' |^2 F_\alpha}{| \alpha' |} + \frac{K^2 v^2 | \alpha' |}{F_\alpha} \,dt}{\int_0^L v^2 F_\alpha | \alpha' | \,dt} < \lambda_{K,N}(\alpha_p)
\]

Now we repeat the argument to obtain
\[
	\lambda_{K,N}(\alpha_p) \le \lambda_{K,N}(\beta)
\]
Let $\epsilon>0$.
There is an $N$-dimensional subspace $W$ of $\Lip_0(0,L)$ such that
\[
	\max_{w \in W} \frac{\int_0^L \frac{|w'|^2 F_\alpha}{| \beta '|} + \frac{K^2 w^2 | \beta '|}{F_\alpha} \,dt}{\int_0^1 w^2 F_\alpha | \beta' | \,dt} < \lambda_{K,N}(\beta) + \epsilon
\]
Let $u$ be a function in $W$ such that
\[
	\frac{\int_0^L \frac{| u' |^2 F_\alpha}{| \alpha_p' |} + \frac{K^2 u^2 | \alpha_p' |}{F_\alpha} \,dt}{\int_0^L u^2 F_\alpha | \alpha_p' | \,dt}
		= \max_{w \in W} \frac{\int_0^L \frac{| w' |^2 F_\alpha}{| \alpha_p' |} + \frac{K^2 w^2 | \alpha_p' |}{F_\alpha} \,dt}{\int_0^L w^2 F_\alpha | \alpha_p' | \,dt}
\]
Note this quantity is at least $\lambda_{K,N}(\alpha_p)$, which is at least $\lambda_{K,N}(\omega)$.
It follows that
\[
	\frac{\int_0^L \frac{| u' |^2 F_\alpha}{| \alpha_p' |} + \frac{K^2 u^2 | \alpha_p' |}{F_\alpha} \,dt}{\int_0^L u^2 F_\alpha | \alpha_p' | \,dt}
		\le \frac{\int_0^L \frac{| u' |^2 F_\alpha}{| \beta' |} + \frac{K^2 u^2 | \beta' |}{F_\alpha} \,dt}{\int_0^L u^2 F_\alpha | \beta' | \,dt}
\]
Now we obtain
\[
	\lambda_{K,N}(\alpha_p) \le \lambda_{K,N}(\beta) + \epsilon
\]
Therefore,
\[
	\lambda_{K,N}(\alpha) < \lambda_{K,N}(\beta)
\]
\end{proof}

Write $\beta=(F_\beta, G_\beta)$.
Define $F_\gamma:[0,L] \to \R$ by
\[
	 F_\gamma(t) =
	 \begin{cases}
		\min \{ F_\beta(s) : s \in [0,t] \} & t \in [0,A] \\
	 	F_\beta & t \in [A,L] \\
	\end{cases}
\]
Let $G_\gamma=G_\beta$.
Let $\gamma=(F_\gamma, G_\gamma)$.
Note that $\gamma:[0,L] \to \R^2_+$ is Lipschitz.

\begin{Lemma}
\label{bg}
Assume $\lambda_{K,N}(\beta) \ge \lambda_{K,N}(\omega)$.
Then
\[
	\lambda_{K,N}(\beta) \le \lambda_{K,N}(\gamma)
\]
\end{Lemma}

\begin{proof}
Define
\[
	V = \bigg\{ t \in [0,A] : F_\beta(t) \neq F_\gamma(t) \bigg\}
\]
By the Riesz sunrise lemma, there are disjoint open intervals $(a_i, b_i)$ such that
\[
	V = \bigcup_i (a_i, b_i)
\]
and $F_\gamma$ is constant over each interval.
Suppose $\lambda_{K,N}(\beta) > \lambda_{K,N}(\gamma)$.
Then there is a $N$-dimensional subspace $W$ of $\Lip_0(0,L)$ such that
\[
	\max_{w \in W} \frac{ \int_0^L \frac{| w' |^2 F_\gamma}{|\gamma'|} + \frac{K^2 w^2 | \gamma'|}{F_\gamma} \,dt}{\int_0^L | w |^2 F_\gamma |\gamma'| \,dt} < \lambda_{K,N}(\beta)
\]
Note that over each interval $(a_i,b_i)$, the function $| \gamma' |$ is zero, so each $w$ in $W$ is constant.
Let $J=[0,L] \setminus V$.
The isolated points of $J$ are countable, so at almost every point in $J$, the curve $\gamma$ is differentiable with $\gamma'=\beta'$.
If $w$ is a non-zero function in $W$, then $w$ cannot vanish identically on $J$, and
\[
	\frac{ \int_J \frac{| w' |^2 F_\beta}{|\beta'|} + \frac{K^2 w^2 | \beta' |}{F_\beta} \,dt}{\int_J | w |^2 F_\beta |\beta'| \,dt}
		= \frac{ \int_0^L \frac{| w' |^2 F_\gamma}{|\gamma'|} + \frac{K^2 w^2 | \gamma' |}{F_\gamma} \,dt}{\int_0^L | w |^2 F_\gamma |\gamma'| \,dt} < \lambda_{K,N}(\beta)
\]
Also for every $w$ in $W$,
\[
	\int_V \frac{| w' |^2 F_\beta}{|\beta'|} + \frac{K^2 w^2 | \beta' |}{F_\beta} \,dt
		= \int_V  \frac{K^2 w^2 | \beta' |}{F_\beta} \,dt \le \lambda_{K,N}(\omega) \int_V | w |^2 F_\beta |\beta'| \,dt
\]
Here the inequality is strict unless $w$ is identically zero over $V$.
It follows that
\[
	\max_{w \in W} \frac{ \int_0^L \frac{| w' |^2 F_\beta}{|\beta'|} + \frac{K^2 w^2 | \beta'|}{F_\beta} \,dt}{\int_0^L | w |^2 F_\beta |\beta'| \,dt} < \lambda_{K,N}(\beta)
\]
This is a contradiction.
\end{proof}

Let $L^*$ be the length of $\gamma$.
Define $\ell:[0,L] \to [0,L^*]$ by
\[
	\ell(t) = \int_0^t | \gamma'(u)| \,du
\]
Define $\rho:[0,L^*] \to [0,L]$ by
\[
	\rho(s) = \min \Big\{ t \in [0,L] : \ell(t)=s \Big\}
\]
This function $\rho$ need not be continuous, but $\zeta = \gamma \circ \rho$ is piecewise continuously differentiable, and for all $t$ in $[0,L]$,
\[
	\zeta(\ell(t)) = \gamma(t)
\]
Morover $\zeta$ is parametrized by arc length.

\begin{Lemma}
\label{gz}
This reparametrization satisfies
\[
	\lambda_{K,N}(\gamma) \le \lambda_{K,N}(\zeta)
\]
\end{Lemma}

\begin{proof}
Write $\gamma=(F_\gamma,G_\gamma)$ and $\zeta=(F_\zeta, G_\zeta)$.
Let $w$ be a function in $\Lip_0(0,L^*)$ such that
\[
	\frac{\int_0^{L^*} \frac{| w' |^2 F_\zeta}{|\zeta'|} + \frac{K^2 w^2 | \zeta' |}{F_\zeta} \,dt}{\int_0^{L^*} | w |^2 F_\zeta |\zeta'| \,dt} < \infty
\]
Define $v=w \circ \ell$.
Then $v$ is in $\Lip_0(0,L)$, and changing variables yields
\[
	\frac{ \int_0^L \frac{| v' |^2 F_\gamma}{|\gamma'|} + \frac{K^2 v^2 | \gamma' |}{F_\gamma} \,dt}{\int_0^L | v |^2 F_\gamma |\gamma'| \,dt}
	= \frac{\int_0^{L^*} \frac{| w' |^2 F_\zeta}{|\zeta'|} + \frac{K^2 w^2 | \zeta' |}{F_\zeta} \,dt}{\int_0^{L^*} | w |^2 F_\zeta |\zeta'| \,dt}
\]
It follows that $\lambda_{K,N}(\gamma) \le \lambda_{K,N}(\zeta)$.
\end{proof}

We can now prove Lemma 1 for the case $K=0$.

\begin{proof}[Proof of Lemma 1 for the case $K=0$]
Suppose $\alpha$ is not equal to $\omega$ and
\[
	\lambda_{K,N}(\alpha) \ge \lambda_{K,N}(\omega)
\]
Then $\alpha$ is not equal to $\beta$, so by Lemmas \ref{ab}, \ref{bg}, and \ref{gz}
\[
	\lambda_{K,N}(\alpha) < \lambda_{K,N}(\beta) \le \lambda_{K,N}(\gamma) \le \lambda_{K,N}(\zeta)
\]
But in this case, $\zeta = \omega$, so the proof is complete.
\end{proof}

For the remainder of the article, we assume that $K$ is positive.
Write $\zeta=(F_\zeta, G_\zeta)$.
Let $P=R-\mu$.
Let $\chi:[0,L^*] \to \R^2_+$ be a piecewise continuously differentiable function such that $\chi(0)=(R,0)$ and for $t$ in $[0,L^*]$ with $t \neq P$,
\[
	\chi'(t) = \Big( F_\zeta'(t), | G_\zeta '(t) | \Big)
\]
Then $\lambda_{K,N}(\zeta)=\lambda_{K,N}(\chi)$, trivially.
Write $\chi=(F_\chi, G_\chi)$.
Note that, for $t$ in $[0,P]$,
\[
	\chi(t) = R-t
\]
Also, for every $t$ in $[0,L^*]$ with $t \neq P$,
\[
	| \chi' | = 1
\]

Let $\Phi_{K,1}, \Phi_{K,2}, \ldots$ be the functions given by Lemma \ref{efs} associated to $\omega$.
Let $z_0$ be the largest root of $\Phi_{K,N}$ in $(0,R)$.
It follows from basic facts about Bessel functions \cite{W} that $z_0 < P$ and that $\Phi_{K,N}$ has no critical points in $[P,R)$.
There is a unique number $\Lambda$ such that there exists a function $u:[z_0,P] \to \R$ which is non-vanishing over $(z_0,P)$ and satisfies
\[
	\begin{cases}
		( \omega u' )' + ( \Lambda \omega - \frac{K^2}{\omega} ) u = 0 \\
		u(z_0) = 0 \\
		u'(P) =  0
	\end{cases}
\]
Moreover, 
\[
	\Lambda < \lambda_{K,N}(\omega)
\]
To compare $\lambda_{K,N}(\chi)$ and $\lambda_{K,N}(\omega)$, we need the following lemma.

\begin{Lemma}
\label{efes}
Let $Q$ and $z$ be real numbers with $z<z_0$ and $Q>P$.
Let $\psi:[z,Q] \to \R^2_+$ be continuously differentiable over $[P,Q]$.
Assume that, for $t$ in $[z,P]$,
\[
	\psi(t) = (R-t,0)
\]
Write $\psi=(F_\psi, G_\psi)$.
Assume that $F_\psi(Q)=0$ and $F_\psi$ is positive over $[z,Q)$.
Assume that $| \psi' |=1$ over $(P,Q)$ and that $F_\psi'(Q) < 0$.
Let $\phi$ be a function in $\Lip_0(z,Q)$ such that
\[
	\lambda_{K,1}(\psi) = \frac{ \int_z^Q | \phi' |^2 F_\psi + \frac{ K^2 \phi^2}{F_\psi} \,dt }{ \int_z^Q \phi^2 F_\psi \,dt }
\]
Assume that $\lambda_{K,1}(\psi) > \Lambda$.
Then
\[
	\lim_{t \to Q} \phi(t) = 0
\]
Also $\phi$ is differentiable over $[z,Q)$, and over $[P,Q)$,
\[
	| \phi' |^2 - \frac{K^2 \phi^2}{|F_\psi|^2 } \le 0
\]
Furthermore $\phi'$ and $\frac{\phi}{F_\psi}$ are bounded over $[z,Q)$.
\end{Lemma}

\begin{proof}
Since $| \phi'|^2 F_\psi$ and $\phi^2/F_\psi$ are integrable, the function $\phi^2$ is absolutely continuous.
Moreover $\phi^2/F_\psi$ is integrable, but $1/F_\psi$ is not integrable over $(c,Q)$ for any $c$ in $(z,Q)$.
It follows that
\[
	\lim_{t \to Q} \phi(t) = 0
\]

By Lemma \ref{efs}, the function $\phi$ is continuously differentiable over $[z,Q)$, and twice continuously differentiable over $[z,P)$ and $(P,Q)$, with
\[
	(F_\psi \phi')' = \frac{K^2 \phi}{F_\psi} -\lambda_{K,N}(\psi) F_\psi \phi
\]
It is also non-vanishing over $(z,Q)$.
We may assume that $\phi$ is positive over $(z,Q)$.
Furthermore, the Picone identity (see, e.g. Zettl \cite{Z}) implies that
\[
	\phi'(P) < 0
\]
The function
\[
	F_\psi^2 | \phi' |^2 - K^2 \phi^2
\]
is differentiable over $(P,Q)$, and its derivative is
\[
	-2\lambda_{K,N}(\psi) F_\psi^2 \phi \phi'
\]
Therefore, we can prove the inequality by showing that
\[
	\lim_{t \to Q}  F_\psi^2 | \phi' |^2 = 0
\]
Note that
\[
	(F_\psi^2 |\phi'|^2)' = 2K^2 \phi \phi' - 2\lambda_{K,N}(\psi) F_\psi^2 \phi \phi'
\]
Since $| \phi'|^2 F_\psi$ and $\phi^2/F_\psi$ are integrable, it follows that $F_\psi^2 | \phi'|^2$ is absolutely continuous.
Moreover, the limit as $t$ tends to $Q$ must be zero, because $F_\psi | \phi'|^2$ is integrable and $1/F_\psi$ is not integrable over $(c,Q)$ for any $c$ in $(z,Q)$.

It remains to show that $\phi'$ and $\frac{\phi}{F_\psi}$ are bounded over $[z,Q)$.
Let $z_*$ be a point in $[P,Q)$ such that over $[z_*,Q)$,
\[
	\frac{K^2}{F_\psi} -\lambda_{K,N}(\psi) F_\psi > 0
\]
Then $\phi'$ cannot vanish in $[z_*,Q)$.
That is $\phi'$ is negative over $[z_*,Q)$.
We have seen that over $(z_*,Q)$,
\[
	K \phi \ge - F_\psi \phi'
\]
Now over $(z_*,Q)$,
\[
	\phi'' \ge - \frac{F_\psi' \phi'}{F_\psi} - \frac{K \phi'}{F_\psi} -\lambda_{K,N}(\psi) \phi
\]
In particular, since $K \ge 1$,
\[
	\liminf_{t \to Q} \phi'' \ge 0
\]
Therefore $\phi'$ is bounded.
Since $F_\psi'(Q)<0$, it follows from Cauchy's mean value theorem that $\frac{\phi}{F}$ is bounded.
\end{proof}

To compare $\lambda_{K,N}(\chi)$ and $\lambda_{K,N}(\omega)$ we will unroll $\chi$ to $\omega$.
The following lemma describes the homotopy more precisely.

\begin{Lemma}
\label{unroll}
Let $\chi_0:[P,L^*] \to \R^2$ be a continuously differentiable curve, parametrized by arc length.
Assume $\chi_0(P)=(\mu,0)$.
Write $\chi_0=(F_0, G_0)$, and assume that $F_0(L^*)=0$ and $F_0'(L^*)=-1$.
Also assume that $F_0$ is positive over $[P,L^*)$ and $G_0'$ is non-negative over $[P,L^*]$.
Define a curve $\chi_1:[P,L^*] \to \R^2$ by
\[
	\chi_1(t) = \Big( R - t, 0 \Big)
\]
Then there is a $C^1$ homotopy $\chi_s:[P,L^*] \to \R^2$ for $s$ in $[0,1]$ with the following properties.
The homotopy fixes $P$, that is $\chi_s(P)=(\mu,0)$ for all $s$ in $[0,1]$.
Each curve in the homotopy is parametrized by arc length, so for all $t$ in $[P,L^*]$ and for all $s$ in $[0,1]$,
\[
	| \chi_s'(t) | = 1
\]
If we write $\chi_s = (F_s, G_s)$, then for all $t$ in $[P,L^*]$ and for all $s$ in $[0,1]$,
\[
	\dot F_s(t) \le 0
\]
Finally, if $L_s^*$ is defined by
\[
	L_s^* = \min \Big\{ t \in [P,L^*] : F_s(t) = 0 \Big\}
\]
then $F_s'(L_s^*) < 0$, for all $s$ in $[0,1]$.
\end{Lemma}

\begin{proof}
Let $h:[0,1] \to \R$ be a continuously differentiable function such that $h(0)=0$, $h'(0)=0$, $h(1)=1$, $h'(1)=0$ and $h'(s) > 0$ for all $s$ in $(0,1)$.
For functions $f_0:[P,L^*] \to \R$ and $f_1:[P,L^*] \to \R$, with $f_0 \ge f_1$, we define a homotopy by
\[
	f_s=(1-h(s))f_0+h(s)f_1
\]
We refer to this homotopy as the monotonic homotopy from $f_0$ to $f_1$ via $h$.

There is a continuous function $\theta_0:[P,L^*] \to [0,\pi]$ such that, for all $t$ in $[P,L^*]$
\[
	\chi_0'(t) = \Big( -\cos \theta_0(t), \sin \theta_0(t) \Big)
\]
Let $\epsilon>0$ be small.
There is a continuous function $\theta_1:[P,L^*] \to [0,\pi]$, which has the following three properties.
First for all $t$ in $[P,L^*]$,
\[
	\theta_0(t) - \epsilon \le \theta_1(t) \le \theta_0(t)
\]
Second $\theta_1$ is continuously differentiable over the set
\[
	\bigg\{ t \in [P,L^*] : \theta_1(t) \in (\pi/4, \pi] \bigg\}
\]
and $\theta_1$ has finitely many critical points in this set.
Third $\pi/2$ is a regular value of $\theta_1$.
We take the monotonic homotopy from $\theta_0$ to $\theta_1$ via $h$.
The set
\[
	\bigg\{ t \in [P,L^*] : \theta_1(t) \ge \pi/2 \bigg\}
\]
consists of finitely many closed intervals $[a_1, b_1], [a_2,b_2], \ldots$, indexed so that $a_i > b_{i+1}$ for all $i$.
Let $U_1$ be a small neighborhood of $[a_1,b_1]$.
Let $\delta_1>0$ be small, and define $\theta_2:[P,L^*] \to \R$ by
\[
	\theta_2(t) =
	\begin{cases}
		\theta_1(t) & t \notin U_1 \\	
		\min(\theta_1(t), \frac{\pi}{2} - \delta_1) & t \in U_1 \\
	\end{cases}
\]
If $U_1$ is sufficiently small, then for sufficiently small $\delta_1$, this function is continous.
Take the monotonic homotopy from $\theta_1$ to $\theta_2$ via $h$.
Repeat this for each of the closed intervals, letting $U_2, U_3, \ldots$ be small neighborhoods of each of the intervals, and letting $\delta_2, \delta_3, \ldots$ be small positive numbers.
This yields finitely many homotopies.
Finally, take the monotonic homotopy from the last function to the constant zero function via $h$.
Let $\tilde\theta_s:[P,L^*] \to [0,\pi]$, for $s$ in $[0,1]$ be the composition of all of these homotopies.
Then define $\chi_s:[P,L^*] \to \R^2$ for $s$ in $[0,1]$ to be the $C^1$ homotopy with $\chi_s(P)=(\mu,0)$ and for all $t$ in $[P,L^*]$,
\[
	\chi_s'(t) = \Big( -\cos \tilde \theta_s(t), \sin \tilde \theta_s(t) \Big)
\]
If the parameters are sufficiently small, then this homotopy satisfies the properties.
\end{proof}

Now we can compare $\lambda_{K,N}(\chi)$ and $\lambda_{K,N}(\omega)$.

\begin{Lemma}
\label{xo}
If $\chi$ is not equal to $\omega$, then
\[
	\lambda_{K,N}(\chi) < \lambda_{K,N}(\omega)
\]
\end{Lemma}

\begin{proof}
Suppose $\lambda_{K,N}(\chi) \ge \lambda_{K,N}(\omega)$.
Let $\phi_{K,1}, \phi_{K,2}, \phi_{K,3},\ldots$ be the functions given by Lemma \ref{efs} associated to the curve $\chi$.
Let $z$ be the largest root of $\phi_{K,N}$.
Define $\chi_0:[z,L^*] \to \R^2_+$ by
\[
	\chi_0 = \chi \Big|_{[z,L^*]}
\]
It follows from Lemma \ref{efs} that
\[
	\lambda_{K,N}(\chi) = \lambda_{K,1}(\chi_0)
\]
Define $\omega_1:[z,R] \to \R^2_+$ by
\[
	\omega_1(t) = (R-t,0)
\]
It follows from the Picone identity that $z<z_0$ and
\[
	\lambda_{K,N}(\omega) \ge \lambda_{K,1}(\omega_1)
\]
Let $\chi_s:[P,L^*] \to \R^2_+$ be the homotopy discussed in Lemma \ref{unroll}.
Extend the domain of each curve $\chi_s$ to $[z,L^*]$, by defining, for all $s$ in $[0,1]$ and for all $t$ in $[z,P]$,
\[
	\chi_s(t) = \chi_0(t) = (R-t, 0)
\]
For $s$ in $[0,1]$, write $\chi_s=(F_s,G_s)$ and define
\[
	L_s^* = \min \bigg\{ t \in [z,L^*] : F_s(t) = 0 \bigg\}
\]
Then define
\[
	\omega_s = \chi_s \Big|_{[z,L_s^*]}
\]
These functions map into $\R^2_+$.
Note $\omega_1$ agrees with the previous defintion and $\omega_0=\chi_0$.
We will show that the function
\[
	s \mapsto \lambda_{K,1}(\omega_s)
\]
is monotonically increasing over $[0,1]$.
We will do this by showing it is continuous and has non-negative lower left Dini derivative at points $\sigma$ in $(0,1]$ where $\lambda_{K,1}(\omega_\sigma) > \Lambda$.

We first show the function
\[
	s \mapsto \lambda_{K,1}(\omega_s)
\]
is lower semicontinuous.
Fix a point $\sigma$ in $[0,1]$ such that
\[
	\liminf_{s \to \sigma} \lambda_{K,1}(\omega_s) < \infty
\]
Let $\{ s_k \}$ be a sequence in $[0, 1]$ converging to $\sigma$ such that
\[
	\lim_{k \to \infty} \lambda_{K,1}(\omega_{s_k}) = \liminf_{s \to \sigma} \lambda_{K,1}(\omega_s)
\] 
By Lemma \ref{efs}, for each $s$ in $[0, 1]$, there is a function $\phi_s$ in $\Lip_0(z,L_s^*)$ such that
\[
	\lambda_{K,1}(\omega_s) = \frac{\int_z^{L_s^*} |\phi_s'|^2 F_s + \frac{K^2 \phi_s^2}{F_s} \,dt}{\int_z^{L_s^*} \phi_s^2 F_s \,dt}
\]
We may assume that each function $\phi_s$ is normalized so that
\[
	\int_z^{L_s^*} | \phi_s |^2 F_s \,dt = 1
\]
For $s$ in $[0,1]$, let $\ell_s:[z,L_\sigma^*] \to [z,L_s^*]$ be a linear function with $\ell_s(z)=z$ and $\ell_s(L_\sigma^*)=L_s^*$.
Define $W_s = \phi_s \circ \ell_s$, for $s$ in $[0,1]$.
Then define $\tau:[0,1] \to \R$ by
\[
	\tau(s) = \frac{\int_z^{L_\sigma^*} |W_s'|^2 F_\sigma + \frac{K^2 W_s^2}{F_\sigma} \,dt}{\int_z^{L_\sigma^*} W_s^2 F_\sigma \,dt}
\]
Changing variables yields
\[
	\tau(s) = \frac{\int_z^{L_s^*} |\ell_s'|^2 |\phi_s'|^2 (F_\sigma \circ \ell_s^{-1})+ \frac{K^2 \phi_s^2}{(F_\sigma \circ \ell_s^{-1})} \,dt}{\int_z^{L_s^*} \phi_s^2 (F_\sigma \circ \ell_s^{-1}) \,dt}
\]
For $s$ in $[0,1]$, define $\Psi_s:[0,L^*] \to \R$ by
\[
	\Psi_s(t) =
	\begin{cases}
		\frac{F_\sigma \circ \ell_s^{-1}(t)}{F_s(t)} & t \in [0,L_s^*) \\
		1 & t \in [L_s^*,L^*]
	\end{cases}
\]
Note that
\[
	\lim_{s \to \sigma} \Psi_s = 1
\]
and the convergence is uniform.
This follows from the fact that the functions
\[
	(s,t) \mapsto F_\sigma \circ \ell_s^{-1} (t)
\]
and
\[
	(s,t) \mapsto F_s(t)
\]
are both differentiable at the point $(\sigma, L_\sigma^*)$ and their derivatives at this point are equal.
Now we see that
\[
	\lim_{s \to \sigma} \int_z^{L_s^*} \phi_s^2 F_s \, dt - \int_z^{L_s^*} \phi_s^2 (F_\sigma \circ \ell_s^{-1}) \,dt =0
\]
Similarly,
\[
	\lim_{k \to \infty} \int_z^{L_{s_k}^*} |\phi_{s_k}'|^2 F_{s_k} \, dt - \int_z^{L_{s_k}^*} |\phi_{s_k}'|^2 (F_\sigma \circ \ell_{s_k}^{-1}) \,dt =0
\]
Also,
\[
	\lim_{k \to \infty} \int_z^{L_{s_k}^*} \frac{K^2 \phi_{s_k}}{F_{s_k}} \, dt - \int_z^{L_{s_k}^*} \frac{K^2 \phi_{s_k}}{(F_\sigma \circ \ell_{s_k}^{-1})} \,dt =0
\]
It follows that
\[
	\lim_{k \to \infty} \Big( \lambda_{K,1}(\omega_{s_k}) - \tau({s_k}) \Big) = 0
\]
Moreover $\tau(s) \ge \lambda_{K,1}(\omega_\sigma)$ for all $s$ in $[\sigma,1]$.
Therefore,
\[
	\liminf_{s \to \sigma} \lambda_{K,1}(\omega_s) \ge \lambda_{K,1}(\omega_\sigma)
\]
This proves that the function
\[
	s \mapsto \lambda_{K,1}(\omega_s)
\]
is lower semicontinuous.

Next we show the function
\[
	s \mapsto \lambda_{K,1}(\omega_s)
\]
is  upper semicontinuous.
Fix a point $\sigma$ in $[0,1]$.
By Lemma \ref{efs}, there is a function $\phi_\sigma$ in $\Lip_0(z,L_\sigma^*)$ such that
\[
	\lambda_{K,1}(\omega_\sigma) = \frac{\int_z^{L_\sigma^*} |\phi_\sigma'|^2 F_\sigma + \frac{K^2 \phi_\sigma^2}{F_\sigma} \,dt}{\int_z^{L_\sigma^*} \phi_\sigma^2 F_\sigma \,dt}
\]
For $s$ in $[0,1]$, let $\ell_s:[z,L_\sigma^*] \to [z,L_s^*]$ be a linear function with $\ell_s(z)=z$ and $\ell_s(L_\sigma^*)=L_s^*$.
Define $V_s = \phi_\sigma \circ \ell_s^{-1}$, for $s$ in $[0,1]$.
Changing variables yields
\[
	\lambda_{K,1}(\omega_\sigma) = \frac{\int_z^{L_s^*} |\ell_s'|^2 |V_s'|^2 (F_\sigma \circ \ell_s^{-1})+ \frac{K^2 V_s^2}{(F_\sigma \circ \ell_s^{-1})} \,dt}{\int_z^{L_s^*} V_s^2 (F_\sigma \circ \ell_s^{-1}) \,dt}
\]
Then define $\Upsilon:[0,1] \to \R$ by
\[
	\Upsilon(s) = \frac{\int_z^{L_s^*} |V_s'|^2 F_s + \frac{K^2 V_s^2}{F_s} \,dt}{\int_z^{L_s^*} V_s^2 F_s \,dt}
\]
For $s$ in $[0,1]$, define $\Psi_s:[0,L^*] \to \R$ by
\[
	\Psi_s(t) =
	\begin{cases}
		\frac{F_\sigma \circ \ell_s^{-1}(t)}{F_s(t)} & t \in [0,L_s^*) \\
		1 & t \in [L_s^*,L^*]
	\end{cases}
\]
As before,
\[
	\lim_{s \to \sigma} \Psi_s = 1
\]
and the convergence is uniform.
Now we see that
\[
	\lim_{s \to \sigma} \int_z^{L_s^*} V_s^2 F_s \, dt - \int_z^{L_s^*} V_s^2 (F_\sigma \circ \ell_s^{-1}) \,dt =0
\]
Similarly,
\[
	\lim_{s \to \sigma} \int_z^{L_s^*} |V_s'|^2 F_s \, dt - \int_z^{L_s^*} |V_s'|^2 (F_\sigma \circ \ell_s^{-1}) \,dt =0
\]
Also,
\[
	\lim_{s \to \sigma} \int_z^{L_s^*} \frac{K^2 V_s}{F_s} \, dt - \int_z^{L_s^*} \frac{K^2 V_s}{(F_\sigma \circ \ell_s^{-1})} \,dt =0
\]
It follows that
\[
	\lim_{s \to \sigma} \Upsilon(s) = \lambda_{K,1}(\omega_\sigma)
\]
Moreover $\Upsilon(s) \ge \lambda_{K,1}(\omega_s)$ for all $s$ in $[0,\sigma]$.
Therefore,
\[
	\limsup_{s \to \sigma} \lambda_{K,1}(\omega_s) \le \lambda_{K,1}(\omega_\sigma)
\]
This proves that the function
\[
	s \mapsto \lambda_{K,1}(\omega_s)
\]
is upper semicontinuous, hence continuous.
We remark that Cheeger and Colding \cite{CC} proved a general theorem regarding continuity of eigenvalues.

Now we show the left lower Dini derivative of the function
\[
	s \mapsto \lambda_{K,1}(\omega_s)
\]
is non-negative at every point $\sigma$ in $(0,1]$ such that $\lambda_{K,1}(\omega_\sigma)> \Lambda$.
Fix $\sigma$ in $(0,1]$ and assume that
\[
	\lambda_{K,1}(\omega_\sigma) > \Lambda
\]
By Lemma \ref{efs}, there is a function $\phi_\sigma$ in $\Lip_0(0,L_\sigma^*)$ such that
\[
	\lambda_{K,1}(\omega_\sigma) = \frac{\int_z^{L_\sigma^*} |\phi_\sigma'|^2 F_\sigma + \frac{K^2 \phi_\sigma^2}{F_\sigma} \,dt}{\int_z^{L_\sigma^*} \phi_\sigma^2 F_\sigma \,dt}
\]
By Lemma \ref{efes},
\[
	\lim_{t \to L_\sigma^*} \phi_\sigma(t) = 0
\]
Also $\phi'$ and $\frac{\phi}{F_\sigma}$ are bounded over $[z,L^*)$.
Over $[P,L_\sigma^*]$,
\[
	| \phi_\sigma' |^2 - \frac{K^2 \phi_\sigma^2}{|F_\sigma|^2 } \le 0
\]
Note that, for $s$ in $[0,\sigma]$,
\[
	L_s^* \ge L_\sigma^*
\]
Define a function $\xi:[0,\sigma] \to \R$ by
\[
	\xi(s) = \frac{\int_z^{L_\sigma^*} | \phi_\sigma' |^2 F_s + \frac{K^2 \phi_\sigma^2}{F_s} \,dt}{\int_z^{L_\sigma^*} | \phi_\sigma |^2 F_s \,dt}
\]
Now $\lambda_{K,1}(\omega_s) \le \xi(s)$ for $s$ in $[0,\sigma]$, and $\lambda_{K,1}(\omega_\sigma)=\xi(\sigma)$.
Also $\xi$ is left differentiable at $\sigma$ with
\[
	\del_- \xi(\sigma)= \frac{\int_P^{L_\sigma^*} ( | \phi_\sigma' |^2 - \frac{K^2 \phi_\sigma^2}{|F_\sigma|^2 } - \lambda_{K,1}(\omega_\sigma) \phi_\sigma^2 ) \dot F_\sigma \,dt}{\int_z^{L_\sigma^*} | \phi_\sigma |^2 F_\sigma \,dt}
\]
The function $\dot F_\sigma$ is non-positive.
That is, $\del_- \xi(\sigma) \ge 0$.
This implies that the lower left Dini derivative of the function
\[
	s \mapsto \lambda_{K,1}(\omega_s)
\]
is non-negative at $\sigma$.
That is, the lower left Dini derivative is non-negative at every point $\sigma$ in $(0,1]$ such that $\lambda_{K,1}(\omega_\sigma) > \Lambda$.
Since the function is also continuous and $\lambda_{K,1}(\omega_0) > \Lambda$, it follows that the function is monotonically increasing.
Moreover, if $\chi$ is not equal to $\omega$, then for some $\sigma$, the function $\dot F_\sigma$ is not identically zero, which yields $\del_+ \xi(\sigma) < 0$.
This implies that the lower left Dini derivative of the function
\[
	s \mapsto \lambda_{K,1}(\omega_s)
\]
is negative at some point in $[0,1]$.
In particular, the function is not constant.
Now
\[
	\lambda_{K,1}(\chi_0) = \lambda_{K,1}(\omega_0) < \lambda_{K,1}(\omega_1)
\]
This yields $\lambda_{K,N}(\chi) < \lambda_{K,N}(\omega)$.
\end{proof}

\begin{proof}[Proof of Lemma 1]
Suppose $\alpha$ is not equal to $\omega$ and $\lambda_{K,N}(\alpha) \ge \lambda_{K,N}(\omega)$.
Then by Lemmas \ref{ab}, \ref{bg}, \ref{gz}, and \ref{xo},
\[
	\lambda_{K,N}(\alpha) \le \lambda_{K,N}(\beta) \le \lambda_{K,N}(\gamma) \le \lambda_{K,N}(\zeta) = \lambda_{K,N}(\chi) \le \lambda_{K,N}(\omega)
\]
Since $\alpha$ is not equal to $\omega$, it must either be the case that $\alpha$ is not equal to $\beta$ or $\chi$ is not equal to $\omega$.
In the first case, the first inequality is strict by Lemma \ref{ab}.
In the second case, the last inequality is strict by Lemma \ref{xo}.
\end{proof}

\end{document}